\documentclass[11pt,twoside]{article}

\title{Local moduli of holomorphic bundles}
\author{E.\ Ballico \and E.\ Gasparim \and T.\ K\"{o}ppe}
\date{}

\usepackage[pdftitle={Local moduli of holomorphic bundles},
            pdfauthor={Edoardo Ballico, Elizabeth Gasparim, Thomas Köppe},
            pdfsubject={Mathematics, Algebraic Geometry}, pdfproducer={}, pdfcreator={},
            pdfpagelayout=TwoColumnRight, colorlinks=true, linkcolor=black, urlcolor=black, citecolor=black]{hyperref}
\usepackage{amsmath,amssymb,amsthm,amscd,calc,fancyhdr,bm,booktabs,mathrsfs}
\usepackage[outer=3cm,inner=3cm]{geometry}

\numberwithin{equation}{section}
\newtheorem{theorem}{Theorem}[section]
\newtheorem{lemma}[theorem]{Lemma}
\newtheorem{corollary}[theorem]{Corollary}
\newtheorem{proposition}[theorem]{Proposition}

\newtheorem*{thm.gammaEfilt}{Theorem \ref{thm.gammaEfilt}}
\newtheorem*{thm.gamma}{Theorem \ref{thm.gamma}}
\newtheorem*{thm.filt-alg}{Theorem \ref{thm.filt-alg}}

\theoremstyle{remark}
\theoremstyle{definition}

\newtheorem{remark}[theorem]{Remark}
\newtheorem{defn}[theorem]{Definition}

\newcommand{\ce}{\mathrel{\mathop:}=}

\newcommand{\abs}[1]{\left\lvert#1\right\rvert}
\newcommand{\op}{\mathcal{O}_{\mathbb{P}^1}}

\DeclareMathOperator{\Ext}{Ext}

\DeclareMathOperator{\Pic}{Pic}

\DeclareMathOperator{\Tot}{Tot}

\DeclareMathOperator{\vdeg}{\mathbf{deg}}
\DeclareMathOperator{\ad}{ad}
\DeclareMathOperator{\rk}{rk}
\DeclareMathOperator{\SHom}{\mathscr{H}{\!\mathit{om}}}
\DeclareMathOperator{\SEnd}{\mathscr{E}{\!\mathit{nd}}}

\newcounter{dummy}

\begin{document}

\maketitle

\refstepcounter{dummy}\addcontentsline{toc}{section}{Abstract}\begin{abstract}\noindent
We study moduli of holomorphic vector bundles on non-compact
varieties. We discuss filtrability and algebraicity of bundles and
calculate dimensions of local moduli. As particularly interesting
examples, we describe numerical invariants of bundles on some local
Calabi-Yau threefolds.
\end{abstract}


\section{Introduction}

This paper is the second part of a long-term project to study bundles
on threefolds, their moduli and how they change under birational
transformations of the base. An essential component within this plan
is the study of the local situation, that is, holomorphic bundles over
open threefolds such as a neighbourhood of a reduced local complete
intersection $Z$ inside a smooth threefold $W$. The case when $Z$ is a
ruled surface was studied in \cite{bg2}. In this paper we focus on the
case when $Z$ is a curve.

For the local situation, one has to make some delicate choices:
algebraic versus holomorphic bundles; a small analytic neighbourhood
versus a Zariski neighbourhood versus a formal neighbourhood of $Z$ in
$W$. Our choice for this paper is the nicest possible situation,
namely when the conormal bundle of $Z$ inside $W$ is ample. By
``nice'' we mean that in this case the aforementioned choices all
amount to the same results (Theorem \ref{thm.filt-alg} together with
the formal principle \ref{thm.exist}); moreover, filtrability implies
that we can study moduli in a very concrete fashion: by considering
$\Ext$-spaces modulo bundle isomorphisms, which immediately produce
quotient stacks, but \emph{a priori} no moduli spaces. In fact, given
the absence of a notion of stability for bundles on non-compact
manifolds, there is no preferred way to obtain moduli spaces out of
these quotients, and we are led to study stratifications via dimensions
and numerical invariants.

Here we discuss dimension of local deformation spaces and provide
concrete examples of bundles on some Calabi-Yau threefolds. The
threefold situation contrasts totally with the two-dimensional
situation, where the components of the local holomorphic Euler
characteristic provided the required stratifications of the local
moduli (cf.\ \cite[Theorem 4.1]{bg1} and \cite[Theorem 4.15]{bgk1}).
Our examples in this paper show that the local holomorphic Euler
characteristic is not a satisfactory invariant for the case of a curve
inside a threefold. In particular, the invariant $w(E) =
h^0\bigl((\pi_* E)^{\vee\vee} \bigl/ \pi_* E\bigr)$, which provided
meaningful geometric information in the case of an exceptional curve
$Z$ in a surface (where $\pi$ is the contraction of $Z$),
vanishes identically for exceptional curves in $n$-folds for $n \geq
3$ (Lemma \ref{lem.nowidth}). We introduce the notion of partial
invariants and tabulate some examples.

\section{Results}

Let $W$ be a connected complex manifold (or smooth algebraic variety)
and $Z$ a curve contained in $W$ that is reduced, connected and
a local complete intersection. Let $\widehat Z$ denote the formal
completion of $Z$ in $W$. Ampleness of the conormal bundle 
(see Definition \ref{def.ample}) has a strong influence on the 
behaviour of bundles on $\widehat Z$. We show:

\begin{thm.filt-alg}
If the conormal bundle $N_{Z,W}^*$ is ample, then every vector bundle
on $\widehat Z$ is filtrable. If in addition $Z$ is smooth, then every
holomorphic bundle on $\widehat Z$ is algebraic.
\end{thm.filt-alg}

It then follows from Peternell's Existence Theorem (Theorem
\ref{thm.exist}) that the same results hold true for bundles on an
analytic neighbourhood of $Z$ in $W$. Alternatively, one can work with
the more concrete cases where $W$ is the total space of a vector
bundle over $Z$; we consider such spaces with $Z \cong \mathbb{P}^1$
in Sections \ref{sec.P1} and \ref{sec.localcy}. It turns out that when
$Z \cong \mathbb{P}^1$ and $N_{Z,W}^*$ is ample, then the dimension
$\gamma(E)$ of the local deformation space of a bundle $E$ on
$\widehat Z$ depends \emph{only} on the restriction of $E$ to $Z$. We
show:

\begin{thm.gamma}
Let $W = \Tot\bigl(\op(-1)^{\oplus n} \bigr)$ and $Z$ the zero
section. Let $E$ be a vector bundle on $W$ (or on $\widehat Z$) such
that $E\rvert_Z$ has splitting type $a_1 \geq \dotsb \geq a_r$. Then
the dimension of the local deformation space at $E$ is
\[ \gamma(E) = \sum_{a_i-a_j>1} \sum _{t=0}^{a_i-a_j-2} (a_i-a_j-1-t) \cdot \binom{t+n-1}{t} \text{ .} \]
\end{thm.gamma}

However, for curves other than $\mathbb{P}^1$ a similar statement is
very far from the truth, that is, the restriction to $Z$ does not
determine the dimension of the deformation space. So the question of
local moduli becomes more complicated as well as more interesting.

In Section \ref{sec.chainsp1} we consider the case of chains of
transversally intersecting $\mathbb{P}^1$'s. We give some numerical
conditions on a bundle $E$ under which the dimension of the
deformation space depends only on the splitting type (Propositions
\ref{a81} and \ref{a82}); but these are very restrictive conditions
relating the splitting type of $E$ and the degree of the conormal
bundle, seldom satisfied, but nevertheless also satisfied by certain
non-split bundles.

For the case of curves of positive genus, the situation is naturally
more complicated. Dimension of the local deformation space of a split
bundle is given by an exact formula (Theorem \ref{thm.gammaEsplit}),
but in full generality, even for bundles whose filtrations have fixed
degrees the dimension of the local moduli still varies.  Note that for
a curve $Z$ of positive genus, the isomorphism type of the normal
bundle $N_{Z,W}$ is no longer determined by the degree of a
filtration, and the dimension of the local deformation space
$\gamma(E)$ of a bundle $E$ on $\widehat Z$ depends both on $F =
E\rvert_Z$ and $N_{Z,W}$. Yet not all hope is lost, and it is still
possible to give an upper bound for $\gamma(E)$, at least in the case
when $N_{Z,W}$ is a general element of the moduli space $M(Z;r,d)$ of
stable vector bundles on $Z$ of rank $r$ and degree $d$. For
$E\rvert_Z = F$ we write
\[ \gamma(F,N) \ce \sum_{t \geq 0} \gamma(F,N,t) \text{ , where \ }
   \gamma(F, N, t) \ce h^1\bigl(Z; \SEnd(F) \otimes S^t(N^*)\bigr) \text{ ,} \]
and show:

\begin{thm.gammaEfilt}
Let $Z$ be a curve of genus $g \geq 2$. Fix a rank-$r$ vector bundle
$F$ on $Z$ and any increasing filtration $\{F_i\}_{i=0}^{r}$ of $F$
such that $F_r = F$, $F_0 = \{0\}$ and each $F_i\bigl/F_{i-1}$ is a
line bundle. Set $a_i \ce \deg(F_i\bigl/F_{i-1})$, $1 \leq i \leq
r$. Assume that $N_{Z,W}$ is a general element of $M(Z; n-1, d)$. Then
for any bundle $E$ on $\widehat Z$ with $E\rvert_Z = F$,
\begin{equation*}
  \gamma(F, N_{Z,W}, t) \leq \sum_{i=1}^{r} \sum _{j=1}^{r} 
  \max \Bigl\{0, \ t\bigl(d + 2g - 2 + a_i - a_j\bigr) \bigl/ r
  + 1 - g\Bigr\} \cdot \binom{t + r - 1}{t}
\end{equation*}
for all $t \geq 1$.
\end{thm.gammaEfilt}

We also point out a rigidity behaviour when $g=1$ (Theorem \ref{g=1});
namely, when $Z$ has genus $1$, if both $E\rvert_Z$ and $N_{Z,W}$
are semi-stable, then $\gamma(E\rvert_Z, N_{Z,W}, t) = 0 $ for all $t>0$.

In Section \ref{sec.localcy} we discuss concrete examples of 
Calabi-Yau threefolds and provide a table of numerical
invariants.

\section{Filtrability and Algebraicity}

Let $W$ be a connected, complex manifold (or smooth algebraic variety)
and $Z \subset W$ reduced, connected and a local complete
intersection in $W$. Let $\mathcal{I}$ be the ideal sheaf of $Z$, then
the conormal sheaf $\mathcal{I}_{Z,W} \bigl/ \mathcal{I}_{Z,W}^2$ is
locally free, and we write $N_{Z,W}$ for its dual. We will assume that
the conormal bundle $N_{Z,W}^*$ is ample. For all integers $m \geq 0$
let $Z^{(m)}$ denote the $m^\text{th}$ infinitesimal neighbourhood of
$Z$ in $W$, i.e.\ the closed analytic subspace of $W$ with
${\mathcal{I}_Z}^{m+1}$ as its ideal sheaf. Thus $Z^{(0)} = Z$. Let
$\widehat Z \ce \varprojlim_m Z^{(m)}$ denote the formal completion of
$Z$ in $W$.

In this section we assume that $Z$ is a curve and that the conormal
bundle $N_{Z,W}^*$ is ample. We recall the definition of ampleness for
a general vector bundle (a.k.a.\ locally free sheaf of finite rank)
according to \cite[p.\ 321]{har_avb}:

\begin{defn}\label{def.ample}
Let $\bigl(X, \mathcal{O}_X\bigr)$ be a complex space and
$\mathcal{E}$ a locally free $\mathcal{O}_X$-module of finite rank
(i.e.\ the sheaf of sections of a holomorphic vector bundle $E$). Then
$\mathcal{E}$ (or $E$) is called \emph{ample} if for every coherent
sheaf $\mathcal{F}$ on $X$ there exists an integer $n$ such that
$\mathcal{G} \ce \mathcal{F} \otimes_{\mathcal{O}_X} S^n(\mathcal{E})$
is generated as an $\mathcal{O}_X$-module by its global sections.

[Recall that this means that there exist sections $s_1, \dotsc, s_p
\in \Gamma(X, \mathcal{G})$ such that for every open $U \subseteq X$,
$\mathcal{G}(U)$ is generated by $\{s_1\rvert_{U}, \dotsc,
s_p\rvert_{U}\}$ as an $\mathcal{O}_X(U)$-module. This is equivalent
by \cite[Proposition 2.1]{har_avb} to every stalk $\mathcal{G}_x$
being generated by $\{s_{1,x}, \dotsc, s_{p,x}\}$ as an
$\mathcal{O}_{X,x}$-module.]

We will also use the term \emph{spanned} for ``generated by global
sections'', which is equivalent on complex spaces.\footnote{The notion
of a \emph{spanned} vector bundle is more general on spaces over
non-algebraically closed fields.}
\end{defn}

The aim of this section is to prove:

\begin{theorem}\label{thm.filt-alg}
If the conormal bundle $N_{Z,W}^*$ is ample, then every vector bundle
on $\widehat Z$ is filtrable. If in addition $Z$ is smooth, then every
holomorphic bundle on $\widehat Z$ is algebraic.
\end{theorem}


\begin{lemma}\label{lem.a1}
Assume that $Z \subset W$ is a smooth curve and that $N_{Z,W}^*$ is
ample. Then there exists $H \in \Pic \widehat Z$ such that
$H\rvert_{Z^{(n)}}$ is ample for all $n$.
\end{lemma}
\begin{proof}
For each $m > n \geq 0$ the restriction map $\rho_{m,n} \colon \Pic
Z^{(m)} \to \Pic Z^{(n)}$ is surjective because $H^2\bigl(\widehat Z; S^i(N^*)\bigr) = 0$
for all $i$; the result follows from the definition of $\Pic \widehat Z$.
\end{proof}

\begin{lemma}[Filtrability]\label{lem.extension2}
Assume that $Z \subset W$ is a smooth curve and that $N_{Z,W}^*$ is
ample. Fix any integer $r \geq 2$ and any rank-$r$ vector bundle $E$
on $\widehat Z$. Then there exists an increasing filtration $E_1 \subset
\dotsb \subset E_{r-1} \subset E_r \ce E$ of $E$ by subbundles such
that $E_1 \in \Pic \widehat Z$ and $E_i \bigl/ E_{i-1} \in \Pic
\widehat Z$ for all $2 \leq i \leq r$.
\end{lemma}
\begin{proof}
Fix $H \in \Pic \widehat Z$ as in Lemma \ref{lem.a1}. Since
$N_{Z,W}^*$ is ample, there is an integer $m$ such that
$S^n(N_{Z,W}^*)$ is spanned (i.e.\ generated by global sections) for
all integers $n \geq m$. Since $H\rvert_{Z^{(m)}}$ is ample, there is
an integer $t \geq 0$ such that $H^1\bigl(Z^{(m)};
E(tH)\rvert_{Z^{(m)}}\bigr) = 0$ and $E(tH)\rvert_{Z^{(m)}}$ is
spanned. Since $\dim Z^{(m)} = 1$, $h^2\bigl(Z^{(m)};
\mathcal{F}\bigr) = 0$ for every coherent sheaf $\mathcal{F}$ on
$Z^{(m)}$ and since $S^n(N_{Z,W}^*)$ is spanned for all integers $n
\geq m$, we get
\[ H^1\Bigl(Z^{(m)}; \bigl(E(tH)\rvert_{Z^{(m)}}\bigr) \otimes S^n(N_{Z,W}^*) \Bigr) = 0 \quad\text{for all \ } n \geq m \]
and that $\bigl(E(tH)\rvert_{Z^{(m)}}\bigr) \otimes S^n(N_{Z,W}^*)$ is spanned
for all $n \geq m$. Using the exact sequences
\begin{equation*}
  0 \longrightarrow S^n(N^*) \longrightarrow \mathcal{O}_Z^{(n)}
  \longrightarrow \mathcal{O}_Z^{(n-1)} \longrightarrow 0 \text{ ,}
\end{equation*}
we get that $E(tH)\rvert_{Z^{(n)}}$ is spanned for all $n \geq m$ and
also that the restriction map
\[ \eta \colon H^0\bigl(\widehat Z; E(tH)\bigr) \to H^0\bigl(Z^{(m)}; E(tH)\rvert_{Z^{(m)}}\bigr) \]
is surjective. Since $\dim Z^{(m)} = 1$ and $r \geq 2$, there is a
nowhere vanishing section
\[ \sigma \in H^0\bigl(Z^{(m)}; E(tH)\rvert_{Z^{(m)}}\bigr) \text{ ,} \]
i.e.\ an inclusion $(H^{\otimes t})^\vee\rvert_{Z^{(m)}} \to
E\rvert_{Z^{(m)}}$ with locally free cokernel. Take $\alpha \in
H^0\bigl(\widehat Z; E(tH)\bigr)$ such that $\eta(\alpha) =
\sigma$. Since $Z$ is the support of both $\widehat Z$ and $Z^{(m)}$,
$\eta$ induces an inclusion $j \colon (H^{\otimes t})^\vee \to E$ with
locally free cokernel. Set $E_1 \ce j\bigl((H^{\otimes
t})^\vee\bigr)$. If $r=2$, then we are done. If $r \geq 3$ we apply
induction on $r$ to the rank-$(r-1)$ vector bundle $E \bigl/ E_1$.
\end{proof}

\begin{theorem}[Algebraicity]\label{thm.alg}
Assume that $Z \subset W$ is a smooth curve and that $N_{Z,W}^*$ is
ample. Then holomorphic vector bundles on $\widehat Z$ are algebraic.
\end{theorem}
\begin{proof}
By Lemma \ref{lem.extension2}, vector bundles on $\widehat Z$ are extensions
\[ 0 \longrightarrow L_{a_r} \longrightarrow E \longrightarrow F_{r-1} \longrightarrow 0 \]
and are therefore classified by $\Ext^1\bigl(F_{r-1}, L_{a_r}\bigr) =
H^1\bigl(F_{r-1} \otimes L^{-1}_{a_r}\bigr)$, which is
finite-di\-men\-sio\-nal. Furthermore, each such extension (and hence each
such $E$) is uniquely determined from some finite, infinitesimal
neighbourhood $X^{(n)}$, an algebraic object, and consequently $E$ is
algebraic.
\end{proof}

\section{Local deformation space}

Let $W$ be a connected, complex manifold (or smooth algebraic variety)
and $X \subset W$ reduced, connected and a local complete
intersection in $W$. We recall some general definitions and properties
of deformation spaces.

\begin{defn}[Deformation of a vector bundle, {\cite[p.\ 113]{pet1}}]
Let $X$ be complex space and $V_0$ a holomorphic vector bundle on $X$. A
\emph{deformation} of $V_0$ is a pair $(S,V)$ consisting of a germ
$(S,s_0)$ of a complex space and a vector bundle $V$ on $X \times S$
together with an isomorphism $V_0 \cong V\rvert_{X\times\{s_0\}}$.

Let $(S',V')$ be another deformation of $V_0$. A \emph{morphism}
$(S',V') \to (S,V)$ is a pair $(\alpha, f)$ consisting of a morphism
of $\alpha \colon S' \to S$ of germs of spaces and a morphism of
vector bundles $f \colon V' \to V$ over $\alpha$ (i.e.\ a holomorphic
map which maps fibrewise and is linear on fibres), which respects the
isomorphisms of the central fibres.

A deformation $(S,V)$ of $V_0$ is called
\begin{enumerate}
\item \emph{complete} if for every deformation $(S',V')$ of $V_0$
      there exists a morphism
      \[ (\alpha,f) \colon (S',V') \to (S,V) \text{ ,} \]
\item \emph{effective} if the following is true:

      If $(S',V')$ is another deformation of $V_0$ and if $(\alpha,
      f), (\beta, g) \colon (S',V') \to (S,V)$ are morphisms of
      deformations, and if $T(\alpha), T(\beta) \colon T(S') \to T(S)$
      are the corresponding tangential maps, then $T(\alpha) =
      T(\beta)$,
\item \emph{semi-universal} if $(S,V)$ is complete and effective.
\end{enumerate}
\end{defn}

In this paper we calculate the minimal dimension of a semi-universal
deformation for bundles on $\widehat X$ in the case when $N_{X,W}^*$
is ample. The existence of a finite-dimensional, semi-universal
deformation in this situation is given by results of Peternell:
  
\begin{remark}\label{rem.dim1}
If the conormal sheaf $N_{X,W}^*$ is ample, the deformation space of a
bundle on $\widehat{X}$ is finite-dimensional:

Fix an integer $m \geq 0$, a vector bundle $E_m$ on $X^{(m)}$ and set
$E_0 \ce E_m\rvert_X$. If
\[ h^2\bigl(X; \SEnd E_0 \otimes S^m(N_{X,W}^*) \bigr) = 0 \text{ ,} \]
then there exists a vector bundle $E_{m+1}$ on $X^{(m+1)}$ such that
$E_{m+1}\vert_{X^{(m)}} \cong E_m$ (\cite[Satz 1]{pet1}). 

Now let $F$ be a vector bundle over $X$ such that $h^2\bigl(X; \SEnd F
\otimes S^t(N_{X,W}^*) \bigr) = 0$ for all $t>0$. If $N_{X,W}^*$ is
ample, then $h^1\bigl(X; \SEnd F \otimes S^t(N_{X,W}^*) \bigr) = 0$
for $t \gg 0$, and hence
\[ \gamma = \sum_{t \geq 0} h^1\bigl( X; \SEnd F \otimes S^t(N_{X,W}^*) \bigr) < +\infty \text{ .} \]
Then there exists a vector bundle $G$ on $\widehat X$ such that
$G\vert_X \cong F$, and for a fixed such $G$ the \emph{deformation
space} of $G$ is isomorphic to $\mathbb{C}^\gamma$ (\cite[Satz 2]{pet2},
and first Bemerkung at p.\ 115, and see also \cite[Theorem
10.3.16]{djp}). There is a vector bundle $A$ on an analytic
neighbourhood $U$ of $X$ in $W$ such that $A\vert_{\widehat X} = G$,
and hence $A\vert_X \cong F$ (\cite[Satz 3]{pet2}).
\end{remark}

We will work with objects on infinitesimal neighbourhoods $X^{(m)}$ or
on the formal completion $\widehat X$ of $X \subset W$. However, in
general the goal is to obtain statements about an actual neighbourhood
of $X$, i.e.\ a tubular neighbourhood. Extending objects that are
defined on the formal neighbourhood to a tubular neighbourhood is done
via the \emph{Formal Principle}.

We restrict our attention to the case when $X$ is an exceptional set,
i.e.\ when $W$ is the resolution $\sigma \colon W \to W'$ of some
(possibly singular) space $W'$, $\sigma(X)=x\in W'$ and $W \backslash
X \cong W'\backslash\{x\}$ via $\sigma$. 

\begin{theorem}[Existence, \cite{pet1}]\label{thm.exist}
Let $W$ be a complex space and let $X$ be an exceptional subspace. Let
$\mathcal{E}$ be a locally free sheaf on the formal completion
$\widehat X$ of $X$. Then there is a locally free sheaf $\mathcal{F}$
on a neighbourhood $U \supset X$ in $W$ such that
$\mathcal{F}\rvert_{\widehat X} \cong \mathcal{E}$. Moreover,
$\mathcal{F}$ is uniquely determined by the germ of the embedding
$X\hookrightarrow W$.\qed
\end{theorem}

\section{Local holomorphic Euler characteristic}


Suppose that $Z$ is an exceptional curve in $W$, so that $Z$ can be
contracted to a point, and let $W'$ be the space obtained from $W$ by
contracting $Z$ to a (singular) point $x$, and let $\pi \colon W \to
W'$ be the contraction. Following Blache \cite[3.9]{bla1}, we define
the \emph{local holomorphic Euler characteristic} of reflexive sheaves
near $Z$.

\begin{defn}\label{def.locholeul}
The \emph{local holomorphic Euler characteristic} of a reflexive sheaf
$\mathcal{F}$ near $Z \subset W$ is
\begin{equation}\label{eq.locholeul}
  \chi\bigl(Z, \mathcal{F}\bigr) \ce \chi\bigl(x, \pi_*\mathcal{F}\bigr) \ce h^0 \bigl(W'; (\pi_*\mathcal{F})^{\vee\vee}\!\!
  \bigl/ \pi_*\mathcal{F}\bigr) + \sum_{i=1}^{r-1} (-1)^{i-1} h^0\bigl(W';
  R^i \pi_*\mathcal{F} \bigr) \text{ ,}
\end{equation}
where $r = \rk\mathcal{F}$.
\end{defn}

Since in our case $W$ is of cohomological dimension $1$, the higher
direct images $R^i\pi_*\mathcal{F}$ in Equation \eqref{eq.locholeul}
are all zero for $i>1$.

\begin{lemma}\label{lem.nowidth}
Let $Z$ be a curve of codimension $n \geq 2$ in $W$. Then for any
reflexive sheaf $\mathcal{F}$ on $W$ we have
\[ h^0 \bigl(W'; (\pi_*\mathcal{F})^{\vee\vee}\!\!\bigl/ \pi_*\mathcal{F}\bigr) = 0 \text{ .} \]
\end{lemma}
\begin{proof}
The map $\pi$ contracts $Z$ to the point $x$, whereas the restriction
$\pi\rvert_{W \setminus Z} \colon W \setminus Z \to W' \setminus
\{x\}$ is an isomorphism. Let $\widetilde U \subseteq W$ be an open
set containing $Z$ and let $U=\pi(\widetilde U)$. For a section
$\sigma \in \Gamma(U\setminus\{x\}, \pi_*\mathcal{F})$ there is thus a
corresponding section $\widetilde\sigma \in \Gamma(\widetilde U
\setminus Z,\mathcal{F})$ such that $\pi_*\widetilde\sigma = \sigma$.

However, $Z$ is of codimension $n>1$ in $W$, so $\widetilde\sigma$ can
be extended uniquely (by Hartogs' Theorem) to a section
$\widetilde\tau \in \Gamma(\widetilde U,\mathcal{F})$ such that
$\widetilde\tau\rvert_{\widetilde U \setminus Z} = \widetilde\sigma$,
and thus the image $\tau \ce \pi_*\widetilde\tau$ is a unique
extension of $\sigma$ over the singular point, i.e.\
$\tau\rvert_{U \setminus \{x\}} = \sigma$. It follows that
$\Gamma(U \setminus \{x\}, \pi_*\mathcal{F}) \cong \Gamma(U,
\pi_*\mathcal{F})$ for all open sets $U \subseteq W'$, making
$\pi_*\mathcal{F}$ a \emph{normal} sheaf in the sense of Barth, and by
\cite[Proposition 1.6]{har_srs1} $\pi_*\mathcal{F}$ is reflexive; i.e.\
$(\pi_*\mathcal{F})^{\vee\vee} \cong \pi_*\mathcal{F}$.
\end{proof}
\begin{remark}
Note that this result is in strong contrast with the case of surfaces
($n=1$), for which $h^0 \bigl(W';
(\pi_*\mathcal{F})^{\vee\vee}\!\!\bigl/ \pi_*\mathcal{F}\bigr)$
attains a wide variety of values (compare with \cite[Theorem
2.16]{bgk1}).
\end{remark}

\begin{corollary}
If the codimension of the curve $Z$ in $W$ is at least $2$, then for
any reflexive sheaf $\mathcal{F}$ on $W$, we have $\chi\bigl(Z,
\mathcal{F}\bigr) = h^0\bigl(W'; R^1\pi_*\mathcal{F}\bigr)$.\qed
\end{corollary}

\section{Total space of \texorpdfstring{$\bm{\mathcal{O}_{\mathbb{P}^1}(-1)^{ \oplus n}}$}{Tot(O(-1) \textcircumflex (+n) )}}\label{sec.P1}

Let $W$ be an $(n+1)$-dimensional complex manifold, $Z \subset W$ with
$Z \cong {\mathbb {P}}^1$ and $N_{Z,W}$ isomorphic to the direct sum of
$n$ line bundles of degree $-1$. In fact, we can assume that $W=
\Tot\bigl( \mathcal{O}_{\mathbb{P}^1}(-1)^{\oplus n} \bigr)$. We
denote by $\widetilde{Z}$ the inverse limit of all complex manifolds
$W'$ for which $W'$ is an open neighbourhood of $Z$ in $W$ and by
$\widehat{Z} = \varprojlim_m Z^{(m)}$ the formal completion of $Z$ in
$W$.

\begin{lemma}\label{lem.pic}
$\Pic\widetilde{Z} \equiv \Pic\widehat{Z} \cong \mathbb{Z}$. More
precisely, every line bundle $L$ on $\widetilde{Z}$ or $\widehat Z$ is
uniquely determined by the integer $\deg(L\vert_Z)$.
\end{lemma}
\begin{proof}
To check the result for $\widehat Z$, it is sufficient to show that
for all integers $i \geq 0$ the restriction map $\rho _i \colon \Pic
Z^{(i+1)} \to \Pic Z^{(i)}$ is bijective. The map $\rho _i$ is
injective (resp.\ surjective) because $H^1\bigl(Z; S^i(N_{Z,W}^*)
\bigr) = 0$ (resp.\ $H^2\bigl(Z; S^i(N_{Z,W}^*)\bigr) = 0$). The
restriction map $\Pic\widetilde{Z} \to \Pic\widehat Z$ is bijective by
Grothendieck's existence theorem (\cite[5.1.4]{gro-EGAIII1}).
\end{proof}

\paragraph{Notation:} Let $L_k$ denote the only line bundle on $\widetilde{Z}$
or $\widehat Z$ such that $\deg(L\vert_Z) = k$. Also, for any
$\mathcal{O}_W$-module $\mathcal{F}$, we write
\[ \mathcal{F}^{(n)} \ce \mathcal{F} \otimes_{\mathcal{O}_W} \bigl( \mathcal{O}_W
   \bigl/ \mathcal{I}^{n+1} \bigr) \text{ ,} \]
where $\mathcal{I}$ is the ideal sheaf of $Z \subset W$. We can think
of $\mathcal{F}^{(n)}$ as the sheaf of sections of $\mathcal{F}$ on
the $n^\text{th}$ infinitesimal neighbourhood of $Z$.

\begin{lemma}\label{lem.zero}
$H^1\bigl(\widetilde{Z}; L_k\bigr) = H^1\bigl(\widehat Z; L_k\bigr) = 0$ for all $k \geq -1$.
\end{lemma}
\begin{proof}
Direct calculation.
\end{proof}

\begin{proposition}\label{prop.split}
Let $E$ be a vector bundle on $\widetilde{Z}$ or $\widehat Z$ such that
$E\rvert_Z$ has splitting type $a_1 \geq \dotsb \geq a_r$ with $a_r \geq
a_1-1$. Then $E \equiv \bigoplus_{i=1}^{r} L_{a_i}$.
\end{proposition}
\begin{proof}
Use Lemmas \ref{lem.extension2} and \ref{lem.zero}.
\end{proof}

\begin{remark}\label{rem.dim2}
By Remark \ref{rem.dim1} the dimension of the local deformation space
of any vector bundle $E$ on $\widehat Z$ is
\[ \gamma = \sum_{t \geq 0} h^1\bigl( Z; \SEnd(E\rvert_Z) \otimes S^t(N_{Z,W}^*) \bigr) < +\infty \text{ .} \]
Now $a_1 \geq \dotsb \geq a_r$ denote the splitting type of
$E\rvert_Z$. Since here $N_{Z,W}^* \cong \op(1)^{\oplus n}$, we have
\[ \gamma(E) =  \sum_{t \geq 0} \binom{t+n-1}{t} h^1\bigl(\mathbb{P}^1; \SEnd(E\rvert_Z)(t) \bigr) \text{ .} \]
Thus the dimension of the deformation space of $E$ depends only on
$n$ and the splitting type of $E\rvert_Z$. By Grothendieck's existence
theorem the same is true with $\widetilde{Z}$ instead of $\widehat Z$.
\end{remark}

We
calculate $\gamma(E)$:
 
\begin{theorem}[Dimension of local deformation space]\label{thm.gamma}
Let $E$ be a vector bundle on $\widetilde{Z}$ or $\widehat Z$ such that
$E\rvert_Z$ has splitting type $a_1 \geq \dotsb \geq a_r$. Then the
dimension of the local deformation space at $E$ is
\[ \gamma(E) = \sum_{a_i-a_j>1} \sum _{t=0}^{a_i-a_j-2} (a_i-a_j-1-t) \cdot \binom{t+n-1}{t} \text{ .} \]
\end{theorem}
\begin{proof}
According to Remark \ref{rem.dim2}, the dimension $\gamma(E)$ does not
depend on the choice of extension $E$, so we can assume $E = L_{a_1}
\oplus \dotsb \oplus L_{a_r}$. In this case,
\[ \SEnd E  = \bigoplus_{1 \leq i,j \leq r} L_{a_j-a_i} \text{ ,} \]
and consequently
\[ h^1\bigl(Z; \SEnd E \otimes S^t(N_{Z,W}^*)\bigr) = \sum_{1 \leq i,j \leq r}
   h^1\bigl(\mathbb{P}^1; L_{a_i-a_j} \otimes S^t(N_{Z,W}^*)\bigr) \text{ .} \]
If $a_j-a_i \geq -1$, then $h^1\bigl(\mathbb{P}^1; L_{a_j-a_i}\bigr) = 0$,
so we have 
\[ \gamma(E) = \sum_{t \geq 0} \ \sum_{a_i-a_j>1} h^1\bigl(\mathbb{P}^1; L_{a_j-a_i} \otimes S^t(N_{Z,W}^*)\bigr) \text{ .} \]
We perform the computations on $W =
\Tot\bigl(\mathcal{O}_{\mathbb{P}^1}(-1)^{\oplus n}\bigr)$ using the
standard coordinate charts $(z, u_1, \dotsc, u_n) \mapsto (z^{-1},
zu_1, \dotsc, zu_n)$. In these charts, the bundle $L_{a_j-a_i}$ has
transition function $z^{a_i-a_j},$ and consequently a non-trivial
$1$-cocycle $\sigma$ on the $t^\text{th}$ infinitesimal neighbourhood
has an expression of the form:
\[ \sigma = \sum_{i_1 + \dotsb + i_n = t} \ \ \sum_{k = a_j - a_i + 1 + t}^{-1}
   \sigma_{i_1 \ldots i_n k} \; z^k u_1^{i_1} \dotsm u_n^{i_n} \text{ ,} \]
and non-zero terms occur only when $t \leq a_i-a_j-2$. Hence, for
each pair $a_i - a_j > 1$ there are $\sum_{t=0}^{a_i-a_j-2} (a_i-a_j-1-t) \cdot \binom{t+n-1}{t}$ terms.
\end{proof}

\section{Local Calabi-Yau threefolds}\label{sec.localcy}

In this section we specialise to the case of local Calabi-Yau
threefolds. We provide only simple examples to give some idea of the
behaviour of the local holomorphic Euler characteristic ($\chi$ for
short). The general \emph{Macaulay~2} algorithm to compute numerical
invariants for bundles on threefolds will appear in the companion paper
\cite{ko1}. Here we tabulate some invariants for the local threefolds
\[ W_i \ce \Tot\bigl(\op(-i) \oplus \op(i-2)\bigr) \text{ .} \]  
The cases $i=1,2,3$ present quite different behaviour and are
particularly interesting from the point of view of birational
geometry, see \cite{ji}. Recall from Equation \eqref{eq.locholeul}
that the local holomorphic Euler characteristic of a rank-$2$ bundle
is by definition $\chi\bigl(Z, \mathcal{F}\bigr) \ce w + h$, where $w
\ce h^0 \bigl(W'; (\pi_*\mathcal{F})^{\vee\vee}\!\!\bigl/
\pi_*\mathcal{F}\bigr)$ and $h \ce h^0\bigl(W'; R^1 \pi_*\mathcal{F}
\bigr)$. However, by Lemma \ref{lem.nowidth}, on threefolds we have $w
\equiv 0$; and even worse, for $i>1$ bundles on $W_i$ have $\chi =
\infty$. To remedy this situation, we introduce some partial
invariants, obtained by restriction to a surface.

\begin{defn}
Let $D_i \cong \Tot\bigl(\op(-i)\bigr)$, considered as the
hypersurface $\Tot\bigl(\op(-i) \oplus 0\bigr)$ of $W_i$. We define
the \emph{partial invariants}:
\begin{eqnarray*}
  h'\bigl(\mathcal F\bigr) &\ce& h\bigl(\mathcal F\rvert_{D_i}\bigr) \text{ ,} \\
  w'\bigl(\mathcal F\bigr) &\ce& w\bigl(\mathcal F\rvert_{D_i}\bigr) \text{ .}
\end{eqnarray*}
\end{defn}

Some comments about the $W_i$ for $i=1,2,3$:
 
\begin{enumerate}
\item $W_1 \ce \Tot\bigl(\op(-1) \oplus \op(-1)\bigr)$

      Note that $W_1$ is just the specialisation to $n=2$ of the
      spaces in the previous section; and appears in the well-known
      diagram of the basic flop: Let $X$ be the cone over the ordinary
      double point defined by the equation $xy-zw=0$ on
      $\mathbb{C}^4$. The basic flop is described by the diagram:
      \[ \begin{matrix} &  \widetilde{X}  & \cr
                   {\,}^{f_1}\!\! \swarrow
                   & &  {\searrow}^{\!f_2}  \cr
                   X^- & & {\ \ \ \ X^+}  \cr
                   {\,}_{\pi_1} \!\! \searrow
                   & & {\swarrow}_{\!\pi_2}\cr
                  & X & \end{matrix} \text{ ,} \]
      where $\widetilde{X} \ce \widetilde{X}_{x,y,z,w}$ is the blow-up
      of $X$ at the vertex $x=y=z=w=0$, $X^- \ce \widetilde{X}_{x,z}$
      is the blow-up of $X$ along $x=z=0$ and $X^+ \ce
      \widetilde{X}_{y,w}$ is the blow-up of $X$ along $y=w=0$. The
      \emph{basic flop} is the transformation from $X^-$ to $X^+$.

      By Theorem \ref{thm.alg} holomorphic bundles on $W_1$ are
      algebraic, and this has the nice consequence that $\chi$ is
      finite over $W_1$.

\item $W_2 \ce \Tot\bigl(\op(-2) \oplus \op(0)\bigr)$

      Note that $W_2 \cong Z_2 \times \mathbb{C}$, and we can contract
      the zero section on the first factor, thus obtaining a singular
      family $X_2 \times \mathbb{C}$, where $X_2$ is the surface
      containing an ordinary double-point singularity defined by
      $xy-z^2=0$ in $\mathbb{C}^3$. Holomorphic bundles on $W_2$ have
      infinite $\chi$, but $h'$ and $w'$ take finite values on
      $W_2$.

      Note that in contrast to $W_1$, there are strictly holomorphic
      (not algebraic) bundles on $W_2$, although every rank-$2$ bundle
      on $W_2$ is still an extension of line bundles.

\item $W_3 \ce \Tot\bigl(\op(-3) \oplus \op(+1)\bigr)$

      Here not even a partial contraction of the zero section is
      possible, but nevertheless we can still calculate partial
      $h'$ and $w'$.

      Again, on $W_3$ there are strictly holomorphic (not algebraic)
      bundles, and moreover, there are (many) rank-$2$ bundles which
      are not extensions of line bundles.
\end{enumerate}

Given a rank-$2$ bundle $E$ on $W_i$ with vanishing first Chern class,
its restriction to the zero section $Z$ determines an integer $j \geq 0$
called the \emph{splitting type} of $E$, such that $E\rvert_{Z} \cong
\op(j) \oplus \op(-j)$. We consider bundles $E$ that are extensions of
line bundles
\[ 0 \longrightarrow \mathcal{O}\bigl(-j\bigr) \longrightarrow E
     \longrightarrow \mathcal{O}\bigl(j\bigr) \longrightarrow 0 \text{ .} \]
Such a bundle is uniquely determined by $j$ and an extension class $p$.
We tabulate the values of $\chi$, $h'$, $w'$.

\begin{center}
\begin{tabular}{lccccccccc}\toprule
 & & $W_1$ & & &$W_2$ & & & $W_3$ \\
 & \hspace*{1em}$\chi$\hspace*{.7em} & \hspace*{.7em}$h'$\hspace*{.7em} & \hspace*{.7em}$w'$\hspace*{1em}
 & \hspace*{1em}$\chi$\hspace*{.7em} & \hspace*{.7em}$h'$\hspace*{.7em} & \hspace*{.7em}$w'$\hspace*{1em}
 & \hspace*{1em}$\chi$\hspace*{.7em} & \hspace*{.7em}$h'$\hspace*{.7em} & \hspace*{.7em}$w'$\hspace*{1em} \\\midrule
\small $j=0$, any $p$ &    $0$      & $0$ & $0$ & $0$      & $0$      & $0$ & $\infty$ & $0$ & $0$ \\
\small $j=1$, any $p$ &    $0$      & $0$ & $1$ & $0$      & $0$      & $0$ & $\infty$ & $0$ & $0$ \\
\small $j \geq 2$, $p=0$ & $f_0(j)$ & $f_1(j)$ & $g_1(j)$ & $\infty$ & $f_2(j)$ & $g_2(j)$ & $\infty$ & $f_3(j)$ & $g_3(j)$ \\
\small $j=3$, $p=u_1$ &    $3$      & $2$ & $1$ & $\infty$ & $2$      & $1$ & $\infty$ & $2$ & $1$ \\
\small $j=4$, $p=z u_1$     & $6$   & $3$ & $1$ & $\infty$ & $3$      & $0$ & $\infty$ & $3$ & $0$ \\
\small $j=4$, $p=z^3 u_1$   & $7$   & $4$ & $6$ & $\infty$ & $3$      & $2$ & $\infty$ & $3$ & $1$ \\
\small $j=4$, $p=z^3 u_1^2$ & $9$   & $5$ & $6$ & $\infty$ & $4$      & $2$ & $\infty$ & $3$ & $2$ \\
\small $j=5$, $p=z u_1$     & $10$  & $4$ & $1$ & $\infty$ & $4$      & $0$ & $\infty$ & $4$ & $0$ \\
\small $j=5$, $p=z^3 u_1$   & $11$  & $5$ & $6$ & $\infty$ & $4$      & $2$ & $\infty$ & $4$ & $1$ \\
\small $j=5$, $p=z^3 u_1^2$ & $16$  & $7$ & $6$ & $\infty$ & $6$      & $2$ & $\infty$ & $5$ & $2$ \\
\bottomrule
\end{tabular}
\end{center}
The formul\ae{} for the split bundles are as follows:
\begin{align*}
  f_0(j) &= \bigl(j^3-j\bigr)\bigl/6 &
  g_1(j) &= \bigl(j^2+j\bigr)\bigl/2 \\
  f_1(j) &= \bigl(j^2-j\bigr)\bigl/2 &
  g_2(j) &= \left\lfloor\textstyle\frac{j}{2}\right\rfloor j - \left\lfloor\textstyle\frac{j}{2}\right\rfloor^2 \\
  f_2(j) &= \left\lfloor \textstyle \frac{j}{2} \right\rfloor \left(j - \left\lfloor \textstyle \frac{j}{2} \right\rfloor \right) &
  g_3(j) &= \textstyle \frac{1}{2} \left\lfloor\textstyle\frac{j}{3}\right\rfloor\left(2j-1-3\left\lfloor\textstyle\frac{j}{3}\right\rfloor\right) \\
  f_3(j) &= \textstyle \frac{1}{2}\left\lfloor \textstyle \frac{j+1}{3} \right\rfloor \left( 2j+1 - 3\left\lfloor \textstyle \frac{j+1}{3} \right\rfloor \right)
\end{align*}

\section{Neighbourhoods of \texorpdfstring{$\mathbb{P}^{\bm 1}$}{P1}-chains}\label{sec.chainsp1}

Fix an integer $a>0$. Since the case $a=1$ was studied in Section
\ref{sec.P1}, we will often silently assume $a \geq 2$. Let $Y_a$ be
the nodal and connected projective curve with $a$ irreducible
components, say $T_1, \dotsc, T_a$, such that $T_i \cong \mathbb{P}^1$
for all $i$ and $T_i \cap T_j = \varnothing$ if $\abs{i-j} \geq 2$,
and $\# (T_j \cap T_{j+1}) = 1$ for all $j=1, \dotsc, a-1$. Notice
that the plurigenus is $p_a(Y_a) = 0$. Let $W$ be an $n$-dimensional
complex manifold (or smooth algebraic variety) and $Z_a \subset W$ a
compact complex subspace such that $Z_a \cong Y_a$. Abusing notation,
we will also call $T_i$ the irreducible components of $Z_a$. Since
$Z_a$ is a local complete intersection in $W$, its conormal sheaf
$N_{Z_a,W}^*$ is locally free.

Every vector bundle on $Y_a$ is isomorphic to a direct sum of line
bundles. Hence set $N_{Z_a,W}^* \cong L_1 \oplus \dotsb \oplus L_{n-1}$
with $L_j \in \Pic Z_a$. For each $1 \leq j \leq n-1$ and each $1 \leq
i \leq a$ set
\[ b_{i,j} \ce \deg (L_j\rvert_{T_i}) \]
and set
\[ b_i \ce \min _{1 \leq j \leq n-1} b_{i,j} \text{ .} \]
The first set of integers may depend upon a choice of the ordering $L_1,
\dotsc, L_{n-1}$ of the indecomposable factors of $N_{Z_a,W}^*$. The
bundle $N_{Z_a,W}^*$ is ample if and only if each $L_j$ is ample, i.e.\
if and only if each $L_j \rvert_{T_i}$ is ample, i.e.\ if and only if
$b_{i,j} > 0$ for all $1 \leq j \leq n-1$ and all $1 \leq i \leq
a$. Notice that the case $b_{i,j} = 1$ for all $i, j$ is the natural
generalisation of Section \ref{sec.P1}. From now on we will assume the
ampleness of $N_{Z_a,W}^*$.

\begin{remark}\label{rem.3.1}
We have $\Pic Z_a \cong \mathbb{Z}^{\oplus a}$. More precisely, for
every ordered set $(d_1, \dotsc, d_a)$ of $a$ integers there is (up to
isomorphisms) a unique line bundle $L$ on $Z_a$ such that
$\deg(L\rvert_{T_i}) = d_i$ for all $i$. Then $L$ is ample if and only
if it is very ample if and only if $d_i > 0$ for all $i$, and $L$ is
spanned if and only if $d_i \geq 0$ for all $i$. Furthermore, if $L$
is spanned, then $h^1\bigl(Z_a; L\bigr) = 0$ and hence $h^0\bigl(Z_a;
L\bigr) = 1 + \sum _{i=1}^{a} d_i$ (Riemann-Roch).
\end{remark}
Set $\vdeg(L) \ce \bigl(\deg(L\rvert_{T_1}), \dotsc,
\deg(L\rvert_{T_a})\bigr) \in \mathbb{Z}^{\oplus a}$.

\begin{remark}\label{rem.3.2}
Since $\dim Z_a = 1$, $h^2\bigl(Z_a; \mathcal{F}\bigr) = 0$ for every
coherent sheaf $\mathcal{F}$ on $Z_a$. We have $h^1\bigl(Z_a; L\bigr)
= 0$ for every line bundle on $Z_a$ such that $\deg(L\rvert_{T_j})
\geq 0$ for at least $a-1$ indices $j \in \{1, \dotsc, a\}$ and
$\deg(L\rvert_{T_i}) \geq -1$ for all $i$ (use $a-1$ Mayer-Vietoris
exact sequences and the cohomology of line bundles on $\mathbb{P}^1$).
\end{remark}

\begin{lemma}\label{lem.pic-chain}
$\Pic \widetilde{Z_a} \equiv \Pic \widehat{Z_a} \cong
\mathbb{Z}^{\oplus a}$. More precisely, every line bundle $L$ on
$\widetilde{Z_a}$ or $\widehat{Z_a}$ is uniquely determined by the
ordered set of a integers $\vdeg(L\rvert_{Z_a})$.
\end{lemma}
\begin{proof}
To check the result for $\widehat {Z_a}$ it is sufficient to show that
for all integers $i \geq 0$ the restriction map $\rho_i \colon \Pic
Z_a^{(i+1)} \to \Pic Z_a^{(i)}$ is bijective. The map $\rho_i$ is
injective (resp.\ surjective), because $H^1\bigl(Z_a;
S^i(N_{Z_a,W}^*)\bigr) = 0$ (resp.\ $H^2\bigl(Z_a;
S^i(N_{Z_a,W}^*)\bigr) = 0$) (Remark \ref{rem.3.2}). The restriction
map $\Pic \widetilde{Z_a} \to \Pic \widehat{Z_a}$ is bijective by
Grothendieck's existence theorem (\cite[5.1.4]{gro-EGAIII1}).
\end{proof}

\paragraph{Notation:}
For each rank-$r$ vector bundle $F$ on $Z_a$ let $a_{i,1}(F) \geq
\cdots \geq a_{i,r}(F)$ be the splitting type of $F\rvert_{T_i}$. Set
$\epsilon_i(F) \ce a_{i,1}-a_{1,r}$, $1 \leq i \leq a$. For each
rank-$r$ vector bundle $E$ on $\widehat{Z_a}$ set $a_{i,h}(E) \ce
a_{i,h}(E\rvert_{Z_a})$, $1 \leq i \leq a$, $1 \leq h \leq r$, and
$\epsilon_i(E) \ce \epsilon_i(E\rvert_{Z_a}) = a_{i,1}(E) -
a_{i,r}(E)$, $1 \leq i \leq a$.

\paragraph{Notation:}
For every integer $i$ such that $1 \leq i \leq a$ set $T[i] \ce T_1
\cup T_i$. For each $1 \leq i \leq n-1$ set $P_i \ce T_i \cap
T_{i+1}$. Set $T[0] \ce \varnothing$.

\begin{lemma}\label{lem.a4}
Let $F$ be a rank-$r$ vector bundle on $Z_a$ such that $a_{i,r} \geq
-1$ for all $i$ and $a_{i,r} \geq 0$ for at least $a-1$ indices. Then
$h^1\bigl(Z_a; F\bigr) = 0$. If $a_{i,r} \geq 0$ for all $i$, then $F$
is spanned.
\end{lemma}
\begin{proof}
Notice that $a_{i,r} \geq -1$ if and only if $h^1\bigl(T_i;
F\rvert_{T_i}\bigr) = 0$ and that $a_{i,r} \geq 0$ if and only if
$F\rvert_{T_i}$ is spanned. In particular $h^1\bigl(T_1;
F\rvert_{T_1}\bigr) = 0$. Hence we may assume $a \geq 2$. Fix an
integer $1 \leq i < a$ and assume $h^1\bigl(T[i]; F\rvert_{T[i]}\bigr)
= 0$. To get the first assertion of the Lemma by induction on $i$ it
is sufficient to prove $h^1\bigl(T[i+1]; F\rvert_{T[i+1]}\bigr) =
0$. To get the second assertion we also need to prove that if
$F\rvert_{T[i]}$ and $F\rvert_{T_i}$ are spanned, then
$F\rvert_{T[i+1]}$ is spanned. Consider the Mayer-Vietoris exact
sequence
\begin{equation}\label{eqa1}
  0 \longrightarrow  F\rvert_{T[i+1]} \longrightarrow F\rvert_{T[i]} \oplus
  F\rvert_{T_{i+1}} \longrightarrow F\rvert_{\{P_i\}} \longrightarrow 0 \text{ .}
\end{equation}
By assumption, $h^1\bigl(T[i]; F\rvert_{T[i]}\bigr) =
h^1\bigl(T_{i+1}; F\rvert_{T_{i+1}}\bigr) = 0$. If $a_{i+1} \geq 0$,
then $F\rvert_{T_i}$ is spanned, and hence the restriction map
\[ \eta \colon H^0\bigl(T[i]; F\rvert_{T[i]}\bigr) \oplus H^0\bigl(T_{i+1}; F\rvert_{T_{i+1}}) \longrightarrow F\rvert_{\{P_i\}} \]
is surjective. The cohomology exact sequence of \eqref{eqa1} gives
$h^1\bigl(T[i+1]; F\rvert_{T[i+1]}\bigr) = 0$. If $a_{i+1, r} = -1$,
then the inductive assumption gives that $F\rvert_{T[i]}$ is
spanned. Hence the first component of $\eta$ is surjective, and
consequently so is $\eta$, and again we obtain $h^1\bigl(T[i+1];
F\rvert_{T[i+1]}\bigr) = 0$. Similarly, fixing any $Q \in T[i+1]$ we
see that if both $F\rvert_{T[i]}$ and $F\rvert_{T_{i+1}}$ are spanned,
then $F\rvert_{T[i+1]}$ is spanned, concluding the inductive proof.
\end{proof}

\begin{lemma}\label{a5}
Let $E$ be a rank-$r$ vector bundle on $\widehat{Z_a}$. Assume that
$b_i \geq \epsilon_i(E) - 1$ for all $1 \leq i \leq a$, and assume the
existence of $b \in \{1, \dotsc, a\}$ such that $b_i \geq
\epsilon_i(E)$ for all $i \in \{1, \dotsc, a\} \backslash \{b\}$ and
all $1 \leq j \leq n-1$. Then the natural restriction map $\rho_E
\colon H^1\bigl(\widehat{Z_a}; E\bigr) \to H^1\bigl(Z_a;
E\rvert_{Z_a})$ is bijective.
\end{lemma}
\begin{proof}
The surjectivity of $\rho_E$ comes just from Remark \ref{rem.dim1}, it
uses only that $\dim Z_a = 1$ and that $X = Z_a$ is reduced and
a local complete intersection in $W$. To check the injectivity of
$\rho_E$ it is sufficient to prove $H^1\bigl(Z_a; S^m(N_{Z_a,W}^*)
\otimes (E\rvert_{Z_a}) \bigr) = 0$ for all $m \geq 1$. Fix an integer
$m \geq 1$ and set $F \ce S^m(N_{Z_a,W}^*) \otimes (E\rvert_{Z_a})$
and $t \ce \binom{n+m-1}{n-1} r$. Notice that $t = \rk F$.  Notice
that $a_{i,t}(F) = mb_i+a_{i,r}$. Since $m \geq 1$, our assumptions
give $a_{i,t}(F) \geq -1$ for all $i$ and $a_{i,t}(F) \geq 0$ for at
least $a-1$ indices $i$. Apply Lemma \ref{lem.a4}.
\end{proof}

\begin{proposition}\label{a6}
Fix any rank-$r$ vector bundle $E$ on $\widehat{Z_a}$ and a
decomposition $E\rvert_{Z_a} \cong L_1 \oplus \dotsb \oplus L_r$ with
$L_h \in \Pic Z_a$, $1 \leq h \leq r$. Let $\widehat{L}_h$ denote the
only line bundle on $\widehat{Z_a}$ such that
$\widehat{L}_h\rvert_{Z_a} \cong L_h$ (Lemma
\ref{lem.pic-chain}). Assume $b_i \geq 2 \epsilon_i(E) - 1$ for all
$i$ and $b_i \geq 2 \epsilon_i(E)$ for at least $a-1$ indices $i$.
Then $E \cong \widehat{L}_1 \oplus \dotsb \oplus \widehat {L}_r$.
\end{proposition}
\begin{proof}
We fix an isomorphism $\alpha_0 \colon E \to L_1 \oplus \dotsb \oplus
L_n$.  Set $A \ce \SHom(E, \widehat{L}_1 \oplus \dotsb \oplus
\widehat{L}_r)$. It is sufficient to show the existence of $\alpha \in
H^0\bigl(\widehat{Z_a}; A\bigr)$ which induces an isomorphism $E \to
\widehat{L}_1 \oplus \dotsb \oplus \widehat{L}_r$, or, using
Nakayama's Lemma, to prove the existence $\alpha \in
H^0\bigl(\widehat{Z_a}; A\bigr)$ such that $\alpha\rvert_{Z_a} =
\alpha_0$. Equivalently, it is sufficient to show the existence of
\[ \Bigl\{ \alpha_m \in H^0\bigl(Z_a^{(m)}; A\rvert_{Z_a^{(m)}}\bigr) \Bigr\}_{m \geq 1}
   \quad\text{such that}\quad \alpha_m\rvert_{Z^{(m-1)}} = \alpha_{m-1} \text{ .} \]
Fix an integer $m \geq 1$ and assume that $\alpha_0, \dotsc,
\alpha_{m-1}$ have been constructed as above. The obstruction to the
existence of $\alpha_m$ is an element of $H^1\bigl(Z_a;
S^m(N_{Z_a,W}^*) \otimes (A\rvert_{Z_a})\bigr)$. Notice that
$\epsilon_i(A) = 2\epsilon_i(E)$ for all $1 \leq i \leq a$, and now
apply Lemma \ref{a5}.
\end{proof}

\begin{remark}\label{a7}
Let $E$ be a vector bundle on $\widehat{Z_a}$. Set $F \ce
E\rvert_{Z_a}$. By Remark \ref{rem.dim1} the cohomology group
$H^1\bigl(\widehat{Z_a}; \SEnd E\bigr)$ is finite-dimensional, and
$H^1\bigl(\widehat{Z_a}; \SEnd E\bigr)$ (resp.\ $H^1\bigl(Z_a; \SEnd
F\bigr)$) is the tangent space to the deformation space of $E$ (resp.\
$F$). Remark \ref{rem.dim1} gives the surjectivity of the restriction
map
\[ \rho_{\SEnd E} \colon H^1\bigl(\widehat{Z_a}; \SEnd E\bigr) \to H^1\bigl(Z_a; \SEnd F\bigr) \text{ .} \]
In Proposition \ref{a81} we will give a sufficient condition for the
injectivity of the map $\rho_{\SEnd E}$. Then in Proposition \ref{a82}
we will give other cases in which we are able to ``partially
control'' $\ker(\rho_{\SEnd E})$.
\end{remark}

The proof of Proposition \ref{a6} gives the following result.

\begin{proposition}\label{a81}
Fix a rank-$r$ vector bundle $E$ on $\widehat{Z_a}$. Assume that $b_i
\geq 2 \epsilon_i(E) - 1$ for all $i$ and assume $b_i \geq 2
\epsilon_i(E)$ for at least $a-1$ indices $i$. Then the restriction
map $\rho_{\SEnd E} \colon H^1\bigl(\widehat{Z_a}; \SEnd E\bigr) \to
H^1\bigl(Z_a; \SEnd(E\rvert_{Z_a})\bigr)$ is bijective.\qed
\end{proposition}

\begin{remark}
Fix an integer $m \geq 0$ and a vector bundle $A$ on $\widehat{Z_a}$. Let
\[ \rho_{A,m} \colon H^1\bigl(\widehat{Z_a}; A\bigr) \to
H^1\bigl(Z_a^{(m)}; A\rvert_{Z_a^{(m)}}\bigr) \text{ .} \] Remark
\ref{rem.dim1} gives the surjectivity of $\rho_{A,m}$ and that the map
$\rho_{A,m}$ is injective if
\[ h^1\bigl(Z_a; S^t(N_{Z,W}^*) \otimes A\bigr) = 0 \quad\text{for all \ } t \geq m+1 \text{ .} \]
Since $b_i\bigl(S^t(N_{Z,W}^*)\bigr) = t b_i$, the proof of Lemma
\ref{a5} gives that $\rho_{A,m}$ is injective if $(m+1) b_i \geq
\epsilon_i(A) - 1$ for all $i$ and $(m+1) b_i \geq \epsilon_i(A)$ for
at least $a-1$ indices $i$. Now we fix a vector bundle $E$ on
$\widehat{Z_a}$ and take $A \ce \SEnd(E\rvert_{Z_a})$. Since
$\epsilon_i(\SEnd E) = 2 \epsilon_i(E)$, we get the following result.
\end{remark}

\begin{proposition}\label{a82}
Fix a rank-$r$ vector bundle $E$ on $\widehat{Z_a}$ and an integer $m
\geq 0$. If $(m+1) b_i \geq 2 \epsilon_i(E) - 1$ for all $i$ and
$(m+1) b_i \geq 2 \epsilon_i(E)$ for at least $a-1$ indices $i$, then
the restriction map
\[ \rho_{\SEnd E,m} \colon H^1\bigl(\widehat{Z_a}; \SEnd E\bigr) \to H^1\bigl(Z_a^{(m)}; \SEnd(E\rvert_{Z_a^{(m)}})\bigr) \]
is bijective.\qed
\end{proposition}

\section{Neighbourhoods of positive-genus curves}\label{sec.g0}

Let $W$ be an $n$-dimensional complex manifold or an $n$-dimensional
smooth algebraic variety and $Z \subset W$ a closed submanifold which
is a smooth and connected curve of genus $g>0$. \emph{We assume that
the conormal bundle $N_{Z,W}^*$ is ample.}

\begin{remark}\label{rem.semistable}
Let $F$ be a semi-stable vector bundle on $Z$. Since we work in
characteristic zero, $\SEnd F \cong \mathcal{O}_Z \oplus \ad F$, and
$\ad F$ is a semi-stable vector bundle on $Z$ of degree zero
(\cite[Theorem 2.5 or 2.6]{mar1}).
\end{remark}

Let $M(Z; r, d)$ denote the moduli space of all stable vector bundles
on $Z$ with rank $r$ and degree $d$.  If $g \geq 2$, then $M(Z; r, d)$
is a non-empty irreducible variety of dimension $r^2(g-1)+1$.

\begin{remark}
Set $d \ce \deg N_{Z,W}$. Since $\dim Z = 1$, $h^2\bigl(Z; \mathcal{F}
\bigr) = 0$ for every coherent sheaf $\mathcal{F}$ on $Z$. We also
recall that $h^0\bigl(Z; A\bigr) = 0$ if $A$ is a semi-stable (resp.\
stable) vector bundle and $\deg A < 0$ (resp.\ $\deg A \leq 0$). Thus
the restriction map $\Pic \widehat Z \to \Pic Z$ is surjective. The
restriction map $\Pic \widehat Z \to \Pic Z$ is injective if
$h^1\bigl(Z; S^t(N_{Z,W}^*)\bigr) = 0$ for all integers $t \geq 1$,
i.e.\ if
\[ \alpha(N_{Z,W}, t) \ce h^0\bigl(Z; S^t(N_{Z,W}) \otimes \omega _Z\bigr) = 0 \]
for all integers $t \geq 1$.

Now we will give some conditions which ensure that $\alpha(N_{Z,W}, t)
= 0$: Let $\mu_+(N_{Z,W})$ denote the maximal slope of a non-zero
subsheaf of $N_{Z,W}$. The ampleness of $N_{Z,W}^*$ implies
$\mu_+(N_{Z,W}) < 0$. Since the tensor product of semi-stable vector
bundles is semi-stable (\cite[Theorem 2.6]{mar1}), we get
$\mu_+(S^t(N_{Z,W})) = t \mu _+(N_{Z,W})$. Hence $\alpha(N_{Z,W}, t) =
0$ if $N_{Z,W}$ is semi-stable and $d < (n-1)(2-2g)$ or $N_{Z,W}$ is
stable and $d \leq (n-1)(2-2g)$. If $N_{Z,W}$ is a general element of
$M(Z; n-1, d)$, then $\alpha(N_{Z,W}, t) = 0$ for all $t \geq 1$ if
and only if $d \leq (n-1)(1-g)$.
\end{remark}

For any vector bundle $F$ on $Z$ and every integer $t \geq 0$ set
\[ \gamma(F, N, t) \ce h^1\bigl(Z; \SEnd F \otimes S^t(N^*)\bigr) \]
and
\[ \gamma(F, N) \ce \sum_{t \geq 0} \gamma(F,N,t) \text{ .} \]
The ampleness of $N_{Z,W}^*$ implies $\gamma(F, N_{Z,W}) <
+\infty$. Furthermore, there is a vector bundle $E$ on $\widehat Z$
such that $E\rvert_Z \cong F$, the deformation space of any such $E$
is smooth and $\gamma(F, N_{Z,W})$ is the dimension of the deformation
space of any such $E$. Hence our main aim will be the computation of
the integers $\gamma(F, N, t)$ for some $N$ and/or some $F$.

\begin{remark}
Set $r \ce \rk F$. We have
\[ \gamma(F, N, 0) = h^1\bigl(Z; \SEnd F \bigr) = r^2(g-1) + h^0\bigl(Z; \SEnd F \bigr) \text{ ,} \]
where $h^0\bigl(Z; \SEnd F\bigr) \geq 1$, with equality if $F$ is
stable. For an arbitrary $F$ we cannot say much about the integer
$h^0\bigl(Z; \SEnd F\bigr)$ and hence about the integer
$\gamma(F,N,0)$.

As an example, assume $r \geq 2$ and $F \cong
\bigoplus_{i=1}^{r} L_i$, where the $L_i$ are line bundles. Thus
\[ h^0\bigl(Z; \SEnd F\bigr) = \sum_{i=1}^{r} \sum_{j=1}^{r}
   h^0\bigl(Z; L_i \otimes L_j^*\bigr) \text{ .} \]
Hence $h^0\bigl(Z; \SEnd F\bigr) \geq r$, with equality if and only if
$h^0\bigl(Z; L_i \otimes L_j^*\bigr) = 0$ for all $i \neq j$. For all
integers $r, a, x$ there exists a decomposable $F$ as above with rank
$r$, degree $a$ and $h^0\bigl(Z; \SEnd F\bigr) \geq x$.
\end{remark}


\begin{theorem}\label{thm.gammaEsplit}
Assume $g \geq 2$ and fix $F \cong \bigoplus_{i=1}^{r} L_i$, where
each $L_i$ is a line bundle of degree $a_i$. Assume that $N$ is a
general element of $M(Z; n-1, d)$. Then
\begin{equation}\label{eq.gammaEsplit}
  \gamma(F, N, t) = \sum_{i=1}^{r} \sum_{j=1}^{r} \max \Bigl\{0, \
  t\bigl(d + 2g - 2 + a_i - a_j\bigr) \bigl/ r + 1 - g\Bigr\} \cdot \binom{t+r-1}{t}
\end{equation}
for all $t \geq 1$.
\end{theorem}
\begin{proof}
Note that $\SEnd F \cong \bigoplus_{i=1}^{r} \bigoplus_{j=1}^{r} L_i
\otimes L_j^*$, that $L_i \otimes L_j^*$ has degree $a_i - a_j$ if $i
\neq j$ and that $L_i \otimes L_j^* \cong \mathcal{O}_Z$ if $i =
j$. For all integers $m, t$ such that $t \geq 0$ set
\[ u_{t, m} \ce \max \Bigl\{0, \ t\bigl(d+2g-2+m\bigr)\bigl/r
   + 1 - g \Bigr\} \cdot \binom{t+r-1}{t} \text{ .} \]
Fix any $M \in \Pic Z$ and take a general $N \in M(Z; n-1, d)$. By
\cite[Theorem 1]{bal1}, we have $h^1\bigl(Z; S^t(N^*) \otimes M\bigr)
= h^0\bigl(Z; S^t(N) \otimes M^* \otimes \omega _Z\bigr) = u_{t,
m}$. Notice that in the decomposition of $\SEnd F$ only finitely many
line bundles appear. Hence $\gamma(F, N, t) = \sum_{i=1}^{r} \sum
_{j=1}^{r} u_{t, a_i - a_j}$.
\end{proof}

We stress that in the statement of Theorem \ref{thm.gammaEsplit} we
first choose the line bundles $L_i$ and then take a general $N$. If we
fix a general $N$, then for all choices of $L_i \in \Pic^{a_i}(Z)$ we
only claim the inequality
\[ \gamma(F, N, t) \leq \sum_{i=1}^{r} \sum_{j=1}^{r} \max \Bigl\{0, \
   t \bigl(d+2g-2+a_i-a_j\bigr)\bigl/r +1-g \bigr\} \cdot \binom{t+r-1}{t} \]
(just use semi-continuity and the proof of Theorem \ref{thm.gammaEsplit}
just given), but it is easy to get examples (say, with $t=1$) in which
this inequality is strict.

\begin{theorem}[Dimension of local moduli]\label{thm.gammaEfilt}
Assume $g \geq 2$. Fix a rank-$r$ vector bundle $F$ on $Z$ and any
increasing filtration $\{F_i\}_{i=0}^{r}$ of $F$ such that $F_r = F$,
$F_0 = \{0\}$ and each $F_i\bigl/F_{i-1}$ is a line bundle. Set $a_i
\ce \deg\bigl(F_i\bigl/F_{i-1}\bigr)$, $1 \leq i \leq r$. Assume that
$N$ is a general element of $M(Z; n-1, d)$. Then
\begin{equation}\label{eq.gammaEfilt}
  \gamma(F, N, t) \leq \sum_{i=1}^{r} \sum _{j=1}^{r} 
  \max \Bigl\{0, \ t\bigl(d+2g-2+a_i-a_j\bigr)\bigl/r + 1 - g\Bigr\} \cdot \binom{t+r-1}{t}
\end{equation}
for all $t \geq 1$.
\end{theorem}
\begin{proof}
If $r=1$, then $F = F_1$ and we may apply Theorem
\ref{thm.gammaEsplit}. Assume $r \geq 2$. By Theorem
\ref{thm.gammaEsplit} it is sufficient to prove the existence of an
integral smooth affine curve $T$, $0 \in T$, and a flat family
$\{F_\lambda \}_{\lambda \in T}$ such that $F_0 \cong
\bigoplus_{i=1}^{r} F_i\bigl/F_{i-1}$ and $F_\lambda \cong F$ for all
$\lambda \in T\backslash\{0\}$. Notice that $F_1 = F_1\bigl/F_0$ is a
line bundle. Call $\epsilon$ the extension
\begin{equation}\label{eq.extEfilt}
  0 \longrightarrow F_1 \longrightarrow F \longrightarrow F\bigl/F_1 \to 0 \text{ .}
\end{equation}
Set $T \ce \mathbb{C}$. For every $\lambda \in \mathbb{C}
\backslash\{0\}$, the extension $\lambda \epsilon$ has as its middle
sheaf space $F_\lambda$ a vector bundle isomorphic to $F$. The
zero-extension of $F\bigl/F_1$ by $F_1$ has $F_1 \oplus F\bigl/F_1$ as
its middle sheaf. Then use induction on $r$.
\end{proof}

\begin{corollary}
Assume $g \geq 2$. Fix any integer $c$ and assume that $(N, F)$ is
general in $M(Z; n-1, d) \times M(Z; r, cr)$. Then
\[ \gamma(F, N, t) \leq r^2 \cdot \binom{t+r-1}{t} \cdot \max \Bigl\{0, \ t\bigl(d+2g-2\bigr)\bigl/r + 1 - g\Bigr\} \text{ .} \]
\end{corollary}
\begin{proof}
Fix any $R \in \Pic^{-c}(Z)$. Since $\SEnd F \cong \SEnd(F \otimes R)$
and $(N, F \otimes R)$ is general in $M(Z; n-1, d) \times M(Z; r, 0)$,
it is sufficient to do the case $c=0$. Apply Theorem
\ref{thm.gammaEfilt} to $\mathcal{O}_Z^{\oplus r}$ taking $F_i \ce
\mathcal{O}_Z^{\oplus i}$ for all $1 \leq i \leq r$. Hence $a_i=0$ for
all $i$. Every vector bundle on $Z$ is a flat limit of a family of stable
vector bundles with the same degree and the same rank
(\cite[Proposition 2.6]{nr}). Apply this observation to
$\mathcal{O}_Z^{\oplus r}$ and use the semi-continuity theorem for
cohomology groups.
\end{proof}

\begin{corollary}
Assume $g \geq 2$. Fix any integer $d$ and assume that $(N, F)$ is
general in $M(Z;n -1, d) \times M(Z; r, a)$. Write $a = rx+y$ with $x,
y \in \mathbb{Z}$ and $0 \leq y < r$. Set $c \ce \min \bigl\{y,
r-y\bigr\}$. Then $\gamma(F, N, t)$, $t \geq 1$, satisfies the
inequality \eqref{eq.gammaEfilt} with $a_i \ce 0$ for $1 \leq i \leq
r-1$ and $a_r \ce c$.
\end{corollary}
\begin{proof}
First assume $c=y$, i.e.\ $0 \leq y \leq \lfloor r/2 \rfloor$. Fix any
$R\in \Pic^{-x}(Z)$ and any $M \in \Pic^y(Z)$. Note that $\SEnd F
\cong \SEnd(F \otimes R)$ and $(N, F \otimes R)$ is general in $M(Z;
n-1, d) \times M(Z; r, y)$, and every vector bundle on $Z$ is a flat limit of
a family of stable vector bundles with the same degree and the same
rank (\cite[Proposition 2.6]{nr}). Apply this observation to $M
\oplus \mathcal{O}_Z^{\oplus (r-1)}$ and use the semi-continuity
theorem for cohomology groups.

Now assume $c = r - y$. Notice that $\SEnd F \cong \SEnd(F^*)$ and
that $-a \equiv c \pmod{r}$. Apply the first part of the proof to the
vector bundle $F^*$.
\end{proof}

\begin{remark}
For any rank-$(r \geq 2)$ vector bundle $F$ there are many filtrations
as in the statement of Theorem \ref{thm.gammaEfilt} with very
different numerical invariants. However, when $r=2$, there is always a
numerically better filtration, as we will soon see. Let $F$ be a
rank-$r$ vector bundle. First assume that $F$ is not semi-stable. Then
$F$ fits in an extension of line bundles
\[ 0 \longrightarrow L_1 \longrightarrow F \longrightarrow L_2 \longrightarrow 0 \]
with $\deg L_1 > \deg L_2$. Furthermore, this extension is unique and
minimises the integer $\abs{\deg L_1 - \deg L_2}$.

Note that the integer $\deg(L_1) + \deg(L_2) = \deg(F)$ is fixed. To
get the best estimate from inequality \eqref{eq.gammaEfilt}, we need
to find an extension \eqref{eq.extEfilt} which minimises the integer
$\abs{\deg L_1 - \deg L_2}$. If $F$ is semi-stable but not stable,
then we may find an extension (Theorem \ref{thm.gammaEsplit}) with
$\deg L_1 = \deg L_2$. Now assume that $F$ is stable. There is a
unique integer $s(F)$ such that $s(F) \equiv \deg F \pmod{2}$, $0 <
s(F) \leq g$, and there is an extension \eqref{eq.extEfilt} with $\deg
L_2 = \deg L_1 + s(F)$ (\cite[Proposition 3.1]{ln}). Such an extension
is not always unique (see e.g.\ \cite[Corollary 4.6 and Theorem
5.1]{ln}), but any of those minimises the integer $\abs{\deg L_1 -
\deg L_2}$.
\end{remark}

\begin{proposition}\label{g=1}
Assume that $g=1$ and that both $E$ and $N_{Z,W}$ are semi-stable and
write $F \ce E\rvert_Z$. Then $\gamma(F, N_{Z,W}, t) = 0$ for all $t \geq
1$, and consequently the deformation space of $E$ coincides with the
deformation space of $E\rvert_Z$.
\end{proposition}
\begin{proof}
By Serre duality we need to prove that for $t\geq 1$
\[ h^0\bigl(Z; S^t(N_{Z,W}) \otimes \SEnd(E\rvert_Z)\bigr) = 0 \text{ .} \]
Since $N_{Z,W}^*$ is ample, we have $\deg\bigl(S^t(N_{Z,W}) \otimes
\SEnd(E\rvert_Z)\bigr) < 0$. By Atiyah's classification of vector
bundles on elliptic curves (\cite{at1}) or by \cite[Theorem
2.6]{mar1}, $S^t(N_{Z,W}) \otimes \SEnd(E\rvert_Z)$ is
semi-stable. Hence $h^0\bigl(Z; S^t(N_{Z,W}) \otimes
\SEnd(E\rvert_Z)\bigr) = 0$ (\cite{at1}).
\end{proof}


\vfill

\noindent
Edoardo Ballico\\
University of Trento, Department of Mathematics\\
I--38050 Povo (Trento), Italia\\
E-mail: \url{ballico@science.unitn.it}

\bigskip
\bigskip
\noindent
Elizabeth Gasparim and Thomas K\"{o}ppe\\
School of Mathematics,
The University of Edinburgh\\
James Clerk Maxwell Building,
The King's Buildings,
Mayfield Road\\
Edinburgh, EH9 3JZ,
United Kingdom\\
E-mail: \url{Elizabeth.Gasparim@ed.ac.uk}\\
E-mail: \url{t.koeppe@ed.ac.uk}

\end{document}